\documentclass[reqno]{amsart}
\usepackage{amssymb}
\usepackage[usenames, dvipsnames]{color}

\usepackage{verbatim}

\theoremstyle{plain}
\newtheorem{theorem}{Theorem}[section]
\newtheorem{lemma}[theorem]{Lemma}

\theoremstyle{definition}

\theoremstyle{remark}

\newcommand{\Div}{\operatorname{div}}
\providecommand{\abs}[1]{\lvert#1\rvert}
\providecommand{\Abs}[1]{\left\lvert#1\right\rvert}

\numberwithin{equation}{section}

\newcommand{\bR}{\mathbb{R}}

\newcommand\cD{\mathcal{D}}

\providecommand{\abs}[1]{\lvert#1\rvert}
\providecommand{\Abs}[1]{\left\lvert#1\right\rvert}

\providecommand{\norm}[1]{\lVert#1\rVert}

\newcommand{\aint}{{\int\hspace*{-4.3mm}\diagup}}
\makeatletter
\def\dashint{\operatorname%
{\,\,\text{\bf-}\kern-.98em\DOTSI\intop\ilimits@\!\!}}
\makeatother

\begin{document}

\title[Interface transmission problems]{A simple proof of regularity for $C^{1,\alpha}$ interface transmission problems}

\author[H. Dong]{Hongjie Dong}
\address[H. Dong]{Division of Applied Mathematics, Brown University, 182 George Street, Providence, RI 02912, USA}
\email{Hongjie\_Dong@brown.edu }

\begin{abstract}
We give a simple proof of a recent result in \cite{CST20} by Caffarelli, Soria-Carro, and Stinga about the $C^{1,\alpha}$ regularity of weak solutions to transmission problems with $C^{1,\alpha}$ interfaces. Our proof does not use the mean value property or the maximum principle, and also works for more general elliptic systems. Some extensions to $C^{1,\text{Dini}}$ interfaces and to domains with multiple sub-domains are also discussed.
\end{abstract}

\maketitle

\section{Introduction and main results}

In a recent paper \cite{CST20}, Caffarelli, Soria-Carro, and Stinga studied the following transmission problem.  Let $\Omega\in \bR^d$ be a smooth bounded domain with $d \ge 2$, and $\Omega_1$
be a sub-domain of $\Omega$ such that $\Omega_1 \subset\subset\Omega$ and $\Omega_2 = \Omega \setminus \overline{\Omega_1}$. Assume that the interfacial boundary $\Gamma(=\partial\Omega_1)$ between $\Omega_1$ and $\Omega_2$ is $C^{1,\alpha}$ for some $\alpha\in (0, 1)$. Consider the elliptic problem with the transmission conditions
\begin{equation}
                                    \label{eq11.47}
\begin{cases}
\Delta u=0 \quad \text{in}\ \ \Omega_1\cup\Omega_2,\\
u=0\quad \text{on}\ \ \partial\Omega,\\
u|^+_\Gamma=u|^-_\Gamma,\quad \partial_\nu u|^+_\Gamma-\partial_\nu u|^-_\Gamma=g,
\end{cases}
\end{equation}
where $g$ is a given function on $\Gamma$, $\nu$ is the unit normal vector on $\Gamma$ which is pointing inside $\Omega_1$, and $u|^+_\Gamma$ and $u|^-_\Gamma$ (and $\partial_\nu u|_\Gamma^+$ and $\partial_\nu u|_\Gamma^-$) are the left and right limit of $u$ (and its normal derivative, respectively) on $\Gamma$ in $\Omega_1$ and $\Omega_2$.
The main result of \cite{CST20} can be formulated as the following theorem.
\begin{theorem}[Caffarelli, Soria-Carro, and Stinga \cite{CST20}]
                    \label{thm1}
Under the assumptions above, for any $g\in C^\alpha(\Gamma)$, there is a unique weak solution $u\in H^1(\Omega)$ to \eqref{eq11.47}, which is piecewise $C^{1,\alpha}$ up to the boundary in $\Omega_1$ and $\Omega_2$ and satisfies
$$
\|u\|_{C^{1,\alpha}(\overline{\Omega_1})}+\|u\|_{C^{1,\alpha}(\overline{\Omega_2})}\le N\|g\|_{C^\alpha(\Gamma)},
$$
where $N=N(d,\alpha,\Omega,\Gamma)>0$ is a constant.
\end{theorem}

The proof in \cite{CST20} uses the mean value property for harmonic functions and the maximum principle together with an approximation argument. We refer the reader to \cite{CST20} for earlier results about the transmission problem with smooth interfacial boundaries. The main feature of Theorem \ref{thm1} is that $\Gamma$ is only assumed to be in $C^{1,\alpha}$, which is weaker than those in the literature.

In this short paper, we give a simple proof of Theorem \ref{thm1}, which does not invoke the mean value property or the maximum principle, and also works for more general elliptic systems in the form
\begin{equation}\label{systems}
\begin{cases}
\mathcal{L}u:=D_{k}(A^{kl}D_{l}u)=\Div F+f\quad \text{in}\  \Omega_1\cup\Omega_2,\\
u=0\qquad \text{on}\  \partial\Omega,\\
u|^+_\Gamma=u|^-_\Gamma,\quad A^{kl}D_l u\nu_k |^+_\Gamma-A^{kl}D_l u\nu_k|^-_\Gamma=g,
\end{cases}
\end{equation}
where the Einstein summation convention in repeated indices is used,
$$
u=(u^{1},\ldots,u^{n})^{\top},\ \ F_{k}=(F_{k}^{1},\ldots,F_{k}^{n})^{\top},\ \ f=(f^{1},\ldots,f^{n})^{\top},\ \ g=(g^{1},\ldots,g^{n})^{\top}
$$
are (column) vector-valued functions, for $k,l=1,\ldots,d$,
$A^{kl}=A^{kl}(x)$ are $n\times n$ matrices, which are bounded and satisfy the strong ellipticity with ellipticity constant $\kappa>0$:
$$
\kappa|\xi|^{2}\leq A_{ij}^{kl}\xi_{k}^{i}\xi_{l}^{j},\quad|A^{kl}|\leq\kappa^{-1}
$$
for any $\xi=(\xi_{k}^{i})\in\mathbb R^{n\times d}$.

\begin{theorem}
                        \label{thm2}
Assume that $\Omega_1$, $\Omega_2$, and $\Gamma$ satisfy the conditions in Theorem \ref{thm1}, $A^{kl}$ and $F$ are piecewise $C^\alpha$ in $\Omega_1$ and $\Omega_2$, $g\in C^\alpha(\Gamma)$, and $f\in L_\infty(\Omega)$. Then there is a unique weak solution $u\in H^1(\Omega)$ to \eqref{systems}, which is piecewise $C^{1,\alpha}$ up to the boundary in $\Omega_1$ and $\Omega_2$ and satisfies
$$
\sum_{j=1}^2\|u\|_{C^{1,\alpha}(\overline{\Omega_j})}
\le N\|g\|_{C^\alpha(\Gamma)}+N\sum_{j=1}^2\|F\|_{C^\alpha(\Omega_j)}+N\|f\|_{L_\infty(\Omega)},
$$
where $N=N(d,n,\kappa,\alpha,\Omega,\Gamma,[A]_{C^\alpha(\Omega_j)})>0$ is a constant.
\end{theorem}

We also consider the transmission problem with multiple disjoint sub-domains $\Omega_1,\ldots,\Omega_M$ with $C^{1,\alpha}$ interfacial boundaries in the setting of \cite{lv,ln}.
As in these papers, we assume that any point $x\in\Omega$ belongs to the boundaries of at most two of the $\Omega_{j}'s$, so that if the boundaries of two $\Omega_{j}$ touch, then they touch on a whole component of such a boundary. Without loss of generality assume that $\Omega_j\subset\subset \Omega,j=1,\ldots,M-1$ and $\partial\Omega\subset\partial\Omega_{M}$. The transmission problem in this case is then given by
\begin{equation}\label{eq2.36}
\begin{cases}
\mathcal{L}u=\Div F+f\quad \text{in}\  \bigcup_{j=1}^M \Omega_j,\\
u=0\qquad \text{on}\ \ \partial\Omega,\\
u|^+_{\partial\Omega_j}=u|^-_{\partial\Omega_j},\ \ A^{kl}D_l u\nu_k |^+_{\partial\Omega_j}-A^{kl}D_l u\nu_k|^-_{\partial\Omega_j}=g_j,\  j=1,\ldots,M-1.
\end{cases}
\end{equation}
In the following theorem, we obtain an estimate which is independent of the distance of interfacial boundaries, but may depend on the number of sub-domains $M$.

\begin{theorem}
                        \label{thm3}
Assume that $\Omega_j$ satisfy the conditions above, $A^{kl}$ and $F$ are piecewise $C^{\alpha'}$ for some $\alpha'\in (0,\alpha/(1+\alpha)]$, $g_j\in C^{\alpha'}(\partial\Omega_j),j=1,\ldots,M-1$, and $f\in L_\infty(\Omega)$. Then there is a unique weak solution $u\in H^1(\Omega)$ to \eqref{eq2.36}, which is piecewise $C^{1,\alpha'}$ up to the boundary in $\Omega_j,j=1,\ldots,M,$ and satisfies
$$
\sum_{j=1}^j\|u\|_{C^{1,\alpha'}(\overline{\Omega_j})}
\le N\sum_{j=1}^{M-1}\|g\|_{C^{\alpha'}(\partial\Omega_j)}+N\sum_{j=1}^M\|F\|_{C^{\alpha'}(\Omega_j)}
+N\|f\|_{L_\infty(\Omega)},
$$
where $N=N(d,n,M,\kappa,\alpha,\Omega_j,[A]_{C^{\alpha'}(\Omega_j)})>0$ is a constant.
\end{theorem}
It is worth noting that in the special case when $A^{\alpha\beta}$ and $F$ are H\"older continuous in the whole domain, by the linearity the result of Theorem \ref{thm3} still holds with $\alpha'=\alpha$.

Our last result concerns the case when the interfaces are $C^{1,\text{Dini}}$, and $A^{kl}$ satisfy the  piecewise $L_2$-Dini mean oscillation in $\Omega$, i.e.,
\begin{align*}
\omega_{A}(r):=\sup_{x_{0}\in \Omega}\inf_{\bar{A}\in\mathcal{A}}
\bigg(\aint_{\Omega_{r}(x_{0})}|A(x)-\bar{A}|^2\ dx\bigg)^{1/2}
\end{align*}
satisfies the Dini condition, where $\Omega_r(x_0)=B_r(x_0)\cap \Omega$ and $\mathcal{A}$ is the set of piecewise constant functions in $\Omega_{j},j=1,\ldots,M$.

\begin{theorem}
                        \label{thm4}
Assume that $\Omega_j$ satisfy the $C^{1,\text{Dini}}$ condition, $A^{kl}$ and $F$ are of piecewise $L_2$-Dini mean oscillation in $\Omega$, $g_j$ is Dini continuous on $\partial\Omega_j,j=1,\ldots,M-1$, and $f\in L_\infty(\Omega)$. Then there is a unique weak solution $u\in H^1(\Omega)$ to \eqref{eq2.36}, which is piecewise $C^{1}$ up to the boundary in $\Omega_j,j=1,\ldots,M$.
\end{theorem}
We note that the piecewise $L_2$-Dini mean oscillation condition is weaker than the usual piecewise Dini continuity condition in the $L_\infty$ sense.

\section{Proofs}

The idea of the proof is to reduce the transmission problem to an elliptic equation (system) with piecewise H\"older (or Dini) nonhomogeneous terms, by solving a conormal boundary value problem. These equations arose from composite material and have been extensively studied in the literature. See, for instance, \cite{lv,ln}, and also recent papers \cite{D12, dx}. We will apply the results in the latter two papers, the proofs of which in turn are based on Campanato's approach.

\begin{proof}[Proof of Theorem \ref{thm2}]
Let $w\in H^1(\Omega)$ be the weak solution to the conormal boundary value problem
\begin{equation}
                                        \label{eq1.45}
\begin{cases}
\Delta w=c\quad &\text{in}\ \ \Omega_1,\\
w_\nu =g\quad &\text{on}\ \ \partial\Omega_1,\\
\int_{\Omega_1}w\,dx=0,
\end{cases}
\end{equation}
where $c=-|\Gamma|^{-1}\int_{\Gamma}g$ is a constant. The existence and uniqueness of such solution $w$ follows from the trace theorem and the Lax--Milgram theorem, and
\begin{equation}
                            \label{eq2.08}
\|w\|_{H^1(\Omega_1)}\le N\|g\|_{L_2(\Gamma)},
\end{equation}
where $N=N(d,\Omega_1)$. Since $g\in C^\alpha(\Omega_1)$, by the classical elliptic theory (see, for instance, \cite[Theorem 5.1]{Lieb87}), we have
\begin{equation}
                                        \label{eq1.46}
\|w\|_{C^{1,\alpha}(\Omega_1)}\le N\|g\|_{C^\alpha(\Gamma)},
\end{equation}
where $N=N(d,\alpha,\Omega_1)$. By using the weak formulation of solutions, from \eqref{eq1.45} it is easily seen that \eqref{systems} is equivalent to
\begin{equation}\label{eq1.50}
\begin{cases}
\mathcal{L}u=\Div \tilde F+\tilde f\quad &\text{in}\ \ \Omega,\\
u=0\quad &\text{on}\ \ \partial\Omega,
\end{cases}
\end{equation}
where
$$
\tilde F=1_{\Omega_1\cup\Omega_2}F-1_{\Omega_1}\nabla w,\quad \tilde f=f+1_{\Omega_1}c.
$$
By the Lax--Milgram theorem, there is a unique solution $u\in H^1(\Omega)$ to \eqref{eq1.50} and
\begin{align}
                        \label{eq2.09}
\|u\|_{H^1(\Omega)}
&\le N\|F\|_{L_2(\Omega)}+\|\nabla w\|_{L_2(\Omega_1)}+\|f\|_{L_2(\Omega)}+\| c\|_{L_2(\Omega_1)},\nonumber\\
&\le N\|F\|_{L_2(\Omega)}+\|g\|_{L_2(\Gamma)}+\|f\|_{L_2(\Omega)},
\end{align}
where we used \eqref{eq2.08} in the second inequality.
Since $\tilde F$ and $A^{\alpha\beta}$ are piecewise $C^\alpha$,  it follows from \cite[Corollary 2 and Remark 3 (ii)]{D12}, \eqref{eq1.46}, and \eqref{eq2.09} that
\begin{align*}
&\sum_{j=1}^2\|u\|_{C^{1,\alpha}(\overline{\Omega_j})}\\
&\le N\|u\|_{L_2(\Omega)}+N\|F-\nabla w\|_{C^\alpha(\Omega_1)}
+N\|F\|_{C^\alpha(\Omega_2)}+\|f+1_{\Omega_1} c\|_{L_\infty(\Omega)}\\
&\le N\|g\|_{C^\alpha(\Gamma)}+N\sum_{j=1}^2\|F\|_{C^\alpha(\Omega_j)}+N\|f\|_{L_\infty(\Omega)}.
\end{align*}
The theorem is proved.
\end{proof}

\begin{proof}[Proof of Theorem \ref{thm3}]
The proof is similar to that of Theorem \ref{thm2}. In each $\Omega_j,j=1,\ldots,M-1$, we find a weak solution to
\begin{equation}
                                        \label{eq1.45b}
\begin{cases}
\Delta w_j=c_j\quad &\text{in}\ \ \Omega_j,\\
\partial_\nu w_j |^+_{\partial\Omega_j}=g_j\quad &\text{on}\ \ \partial\Omega_j,\\
\int_{\Omega_j}w_j\,dx=0,
\end{cases}
\end{equation}
where $c_j=-|\partial\Omega_j|^{-1}\int_{\partial\Omega_j}g_j$, and $w_j$ satisfies
\begin{equation}
                                        \label{eq1.46b}
\|w_j\|_{C^{1,\alpha'}(\Omega_j)}\le N\|g_j\|_{C^{\alpha'}(\partial\Omega_j)}.
\end{equation}
By using the weak formulation of solutions, it is easily seen that \eqref{eq2.36} is equivalent to
\begin{equation}\label{eq1.50b}
\begin{cases}
\mathcal{L}u=\Div \tilde F+\tilde f\quad &\text{in}\ \ \Omega,\\
u=0\quad &\text{on}\ \ \partial\Omega,
\end{cases}
\end{equation}
where
$$
\tilde F=1_{\cup_{j=1}^M \Omega_j}F-\sum_{j=1}^{M-1}1_{\Omega_j}\nabla w_j,\quad
\tilde f=f+\sum_{j=1}^{M-1}1_{\Omega_j}c_j.
$$
As before, by the Lax--Milgram theorem, there is a unique solution $u\in H^1(\Omega)$ to \eqref{eq1.50b}. Since $\tilde F$ and $A^{\alpha\beta}$ are piecewise $C^{\alpha'}$ and $\partial\Omega_j$ is piecewise $C^{1,\alpha}$, by using \eqref{eq1.46b} and appealing to \cite[Corollary 1.2 and Remark 1.4]{dx}, we conclude the proof of the theorem.
\end{proof}

Finally, we give

\begin{proof}[Proof of Theorem \ref{thm4}]
We claim that under the conditions of the theorem, if $w_j$ is the solution to \eqref{eq1.45b}, then $D w_j$ satisfies the $L_2$-Dini mean oscillation condition in $\Omega_j$. Assuming this is true, then the conclusion of the theorem follows from the proof of Theorem \ref{thm3} and \cite[Theorem 1.1]{dx}. We remark that the $C^1$ continuity of $w_j$ was proved in \cite[Theorem 5.1]{Lieb87} for more general quasilinear equations, but in general $D w_j$ may not be Dini continuous in the $L_\infty$ sense.

To prove the claim, we follow the argument in the proof of Theorem 1.7 of \cite{DLK}. We only give the boundary estimate since the corresponding interior estimate is simpler. By using the $C^{1,\text{Dini}}$ regularity of $\Omega_j$ and locally flattening the boundary, it then suffices to verify Lemma \ref{lem1} below.
\end{proof}

In the sequel, we denote $x=(x',x^d)$, where $x'=(x_1,x_2,\ldots,x_{d-1})\in \bR^{d-1}$, and $\Gamma_r(x):=B_r(x)\cap \{x_d=0\}$ for $x\in \bR^d$ and $r>0$.
\begin{lemma}
                                    \label{lem1}
Let $u\in H^1(B_4^+)$ be a weak solution to
$$
D_k(a^{kl}D_l u)=0\quad \text{in}\ \  B_4^+
$$
with the conormal boundary condition $a^{dl}D_l u=g(x')$ on $\Gamma_4=B_4\cap \{x_d=0\}$, where $a^{kl}=a^{kl}(x)$ satisfy the uniform ellipticity condition and are of $L_2$-Dini mean oscillation, and $g$ is a Dini continuous function on $\Gamma_4$. Then $Du$ is of $L_2$-Dini mean oscillation in $\overline{B_1^+}$.
\end{lemma}
\begin{proof}
We set
\[
g^d(x)= g^d(x',x^d):=g(x'),
\]
which satisfies $D_d g^d=0$.
Therefore, the above problem is reduced to the standard conormal boundary problem
$$
\begin{cases}
D_k(a^{kl}D_l w)=D_d g^d\quad &\text{in}\ \  B_4^+\\
a^{dl}D_l w=g^d \quad&\text{on}\ \ \Gamma_4.
\end{cases}
$$
Similar to \cite[Section 3]{DLK}, for $x\in \overline{B_3^+}$ and $r\in (0,1)$, we define
\[
\phi(x,r) :=\left(\aint_{B_r(x)\cap B_4^+} |Du-(Du)_{B_r(x)\cap B_4^+}|^2 \right)^{\frac1 2},
\]
where
$$
(Du)_{B_r(x)\cap B_4^+}=\aint_{B_r(x)\cap B_4^+} Du.
$$
Fix a smooth domain $\cD$ satisfying
\begin{equation*}
B_{1/2}^{+} \subset \cD \subset B_1^{+}
\end{equation*}
and for $\bar x \in \partial \bR^d_{+}$, we set $\cD_r(\bar x)= r \cD+ \bar x$.
We decompose $u=w+v$, where $w \in H^{1}(\cD_{r}(\bar x))$ is a weak solution of the problem
\begin{align*}
\begin{cases}
D_k(\bar a^{kl} D_l w) = - D_k((a^{kl} - \bar a^{kl})D_l u) +D_d(g^d - \bar{g}^d) \quad &\text{in} \ \ \cD_r(\bar x),\\
\bar a^{kl} D_l w \nu_k = -(a^{kl} - \bar a^{kl})D_l u \nu_k +(g^d - \bar{g}^d)\nu_d \quad &\text{on}\ \ \partial \cD_r(\bar x),
\end{cases}
\end{align*}
where $\bar a^{kl}$ and $\bar g^d$ are the average of $a^{kl}$ and $g^d$ in $\cD_{r}(\bar x)$, respectively. By the $H^1$-estimate, we have
\begin{equation}				\label{eq1323th}
\left(\aint_{B_r^{+}(\bar x)} \abs{Dw}^2\right)^{1/2} \le N\omega_{A}(2r)\,\norm{Du}_{L^{\infty}(B^{+}_{2r}(\bar x))}+N\omega_{g}(2r).
\end{equation}
Note that $v:= u-w$ satisfies
\[
D_k( \bar{a}^{kl} D_l v)= D_d  {\bar g^d} \;\mbox{ in } \; B_r^{+}(\bar x),
\quad \bar {a}^{dl} D_l v = \bar{ g}^d \;\mbox{ on }\; \Gamma_r(\bar x).
\]
Then for any $c \in \bR$ and $k=1, 2, \ldots, d-1$, $\tilde v:= D_k v - c$
satisfies
\[
D_k( \bar{  a}^{kl} D_l \tilde v)= 0 \;\mbox{ in } \;  B_r^{+}(\bar x), \quad \bar {a}^{dl} D_l \tilde v= 0 \;\mbox{ on }\; \Gamma_r(\bar x).
\]
By the standard elliptic estimates for equations with constant coefficients and zero conormal boundary data, we have
for any $c\in \bR$,
\[
\norm{DD_kv}_{L^\infty(B_{r/2}^+(\bar x))} \le Nr^{-1}\left( \aint_{B_r^+(\bar x)} \abs{D_{k}v-c}^2 \right)^{1/2},\quad  k=1, \ldots, d-1.
\]
Then by using $\displaystyle D_{dd}v=-\frac{1}{\bar{a}^{dd}}\sum_{(i,j)\neq (d,d)} \bar{a}^{ij}D_{ij}v$,
we obtain
\[
\norm{D^2v}_{L^{\infty}(B_{r/2}^+(\bar x))} \le N \norm{D D_{x'} v}_{L^{\infty}(B_{r/2}^+(\bar x))} \le Nr^{-1}\left( \aint_{B_r^+(\bar x)} \abs{D_{x'}v- c}^2 \right)^{1/2},
\]
where we used the notation $D_{x'} v= (D_1 v, \ldots, D_{d-1} v)$.
Therefore, we have
\[
\norm{D^2v}_{L^{\infty}(B_{r/2}^+(\bar x))} \le Nr^{-1}\left( \aint_{B_r^+(\bar x)} \abs{Dv- q}^2 \right)^{1/2},\quad \forall  q \in \bR^{d}.
\]
Let $\mu \in (0,1/2)$ be a small number.
Since
\[
\left( \aint_{B_{\mu r}^+(\bar x)} \Abs{Dv-(Dv)_{B_{\mu r}^+(\bar x)}}^2 \right)^{1/2} \le 2\mu r \norm{D^2 v}_{L^{\infty}(B_{\mu r}^+(\bar x))},
\]
we see that there is a constant $N_0=N_0(d, \kappa)>0$ such that
\[
\left( \aint_{B_{\mu r}^+(\bar x)} \Abs{Dv-(Dv)_{B_{\mu r}^+(\bar x)}}^2 \right)^{1/2} \le N_0\mu \left( \aint_{B_r^+(\bar x)} \abs{Dv- q}^2 \right)^{1/2},\quad \forall  q \in \bR^d.
\]
By using the decomposition $u = v+w,$ we obtain from the above and the triangle inequality that
\begin{align*}
&\left( \aint_{B_{\mu r}^+(\bar x)}\Abs{Du-(Dv)_{B_{\mu r}^+(\bar x)}}^2 \right)^{1/2}\\
&\le \left( \aint_{B_{\mu r}^+(\bar x)} \Abs{Dv-(Dv)_{B_{\mu r}^+(\bar x)}}^2 \right)^{1/2}+
\left( \aint_{B_{\mu r}^+(\bar x)} \abs{Dw}^2 \right)^{1/2}\\
&\le N_0 \mu \left( \aint_{B_r^{+}(\bar x)} \abs{Du- q}^2 \right)^{1/2}+N\mu^{-d/2}\left( \aint_{B_r^{+}(\bar x)} \abs{Dw}^2 \right)^{1/2}. 	\end{align*}
By setting $ q =(Du)_{B^+_r(\bar x)}$ and using \eqref{eq1323th}, we obtain
\begin{equation}				\label{eq0750f}
\phi(\bar{x},\mu r) \le N_0 \mu\,\phi(\bar x, r) + N\mu^{-{d}/2}\left( \omega_{{A}}(2r)\norm{Du}_{L^{\infty}(B_{2 r}^+(\bar x))}+\omega_{ g}(2r) \right).
\end{equation}
By using an iteration argument as in the proof of \cite[Theorem 1.7]{DLK}, from \eqref{eq0750f} and the corresponding interior estimate, it is easily seen that $Du$ is of $L_2$-Dini mean oscillation in $\overline{B_1^+}$ with a modulus of continuity depending on $d$, $\kappa$, $\|Du\|_{L_2(B^+_4)}$, $\omega_g$, and $\omega_A$. The lemma is proved.
\end{proof}

\end{document}